\documentclass[11pt]{amsart}
 \usepackage{amssymb}
\usepackage{amsthm}
\usepackage{xspace}
\usepackage{amsmath}
\usepackage{setspace}
\usepackage{amscd}
\usepackage{tikz}
\usepackage{setspace}
\usepackage{enumerate}
\usepackage{graphicx}

\makeatletter
\@namedef{subjclassname@2010}{%
  \textup{2010} Mathematics Subject Classification}
\makeatother

\newtheorem{theorem}{Theorem}[section]
\newtheorem{lemma}[theorem]{Lemma}
\newtheorem{proposition}[theorem]{Proposition}
\theoremstyle{definition}

\theoremstyle{remark}

\numberwithin{equation}{section}
 
\newcommand{\Z}{\mbox{$\mathbb{Z}$}}

\newcommand{\Q}{\mbox{$\mathbb{Q}$}}

\begin{document}

\title{ Primes  of the form $x^2+dy^2$  with  \\ $x\equiv 0\pmod{N}$ or $y\equiv 0\pmod{N}$ }


\author{Ambedkar Dukkipati}
\address{Department of Computer Science and Automation,\\ Indian Institute of  Science,
Bangalore, India. }
\email{ ad@csa.iisc.ernet.in}

\author{Sushma Palimar}\thanks{}
\address{Department of Computer Science and Automation,\\ 
Indian Institute of  Science,
Bangalore, India.}

 \email{sushma@csa.iisc.ernet.in, sushmapalimar@gmail.com.}





 \begin{abstract}
In this paper we charatcterize primes of the form $x^2+dy^2$  with  $x\equiv 0\pmod{N}$ or
 $y\equiv 0\pmod{N}$ for positive integer $N$ and $d$ with $d$ being square free.
\end{abstract}




\maketitle
   
\section{Introduction}	
The problem of representing primes in the form $x^2+dy^2$ is ancient, 
deep and is    prelude to class field theory.
 For $d=1$, we have Fermat's  celebrated theorem in number theory, 
 which states that, for 
\[p>2 \quad p=x^2+y^2 \Leftrightarrow p\equiv 1\pmod{4}.\] 
This lead many mathematicians to work on primes of the form $x^2+dy^2$ for
$d = 2,3,5,7$ and so on. Cox \cite{Cox:1989:pfct}  solved this problem   
using the techniques from class field theory. 
Cox's result is as follows: 
 let $n$ be a positive integer, $p$ an odd prime relatively prime to $d$ and
 let $\mathcal{O}=\Z[\sqrt{-n}]$ be the order in 
 $K=\Q(\sqrt{-n})$. 
 Then \(p=x^2+ny^2 \text{ if and only if  } p\){ splits completely in the ring class field }$K_{\mathcal{O}}.$
 Let $f_{n}(x)\in \Z[X]$ be the minimal polynomial of a real algebraic integer which
 generates $K_{\mathcal{O}}$ over $K$, and $p$ is relatively prime to the discriminant of $f_n(x)$.
 Then \(p=x^2+ny^2 \text{ if and only if } \left(\frac{-n}{p}\right)=1 \text{ and } 
 f_n(x)\equiv 0\pmod {p} \text{ has an integer solution}. \)
 A finite extension $L$ of  $K$ is called unramified (at the finite primes) if $\mathrm{disc}(L/K)=1$; 
 and abelian if $L/K$ is normal  with abelian Galois group.
 The maximal unramified abelian extension $H$ of 
$K$ is finite and is called the Hilbert class field of $K$. By Artin's reciprocity law, we have 
$\mathrm{Gal}(H/K)=\mathrm{Cl}(K)$, where $\mathrm{Cl}(K)$ is the class group of $K$ and $\mathrm{Gal}(H/K)$ 
is the Galois group
of $H$ over $K$. Let $L/K$ be a Galois extension of number fields, with $G=\mathrm{Gal}(L/K)$.
Let $I_{K}(\Delta(L/K)$ be the group of fractional ideals relatively prime to the discriminant $\Delta(L/K)$
so that every prime ideal in $I_{K}(\Delta(L/K)$ is unramified in $L$.
Then we have the Artin map \cite{lan:1994:ant},
\(I_{K}(\Delta(L/K)\mapsto \mathrm{Gal}(L/K), \text{  which is a surjective homomorphism.}\)
Let $m$ be a modulus of a number field $K$. We define $P_{K,1}$ to be the subgroup of $I_{K}(m)$
generated by the principal ideals $\alpha\cdot\mathcal{O}_{K}$, with $\alpha\in \mathcal{O}_{K}$, satisfying
$\alpha\equiv1\pmod{m}$ and $\sigma(\alpha)>0$ for all infinite primes $\sigma$ of $K$.
We define $\mathrm{Cl}_{m}(K)$ to be the group $I_{K}(m)/P_{K,1}(m)$.
Let $m$ be a modulus divisible by all primes which ramify in $L$. Then the kernel of the 
Artin map $I_{K}(m)\mapsto \mathrm{Gal}(L/K)$ contains $P_{K,1}(m)$ if and only if $f|m$, where $f$ is the conductor of $L/K$
\cite{Cox:1989:pfct}.

A Mersenne prime $M_p=2^p-1$ is a quadratic residue of $7$ if and only if $p\equiv 1\bmod{3}$.
 In such cases $M_p\equiv 1\bmod{7}$ and clearly
$2^p-1=x^2+7y^2$ for some integers  $x$ and $y$.  Lemmermeyer made an observation that,  
 $x$ is divisible by $8$ and $y$ leaves the remainder $\pm 3\bmod{8}$, which was later proved by Lenstra and
 Stevenhagen in   \cite{henstev:2000:artinreci} using Artin reciprocity law.
    
The notion of Mersenne primes in the rings of integers of real quadratic fields
 $\Q(\sqrt{d})$, defined by $M_{p,\alpha}=\frac{{\alpha}^{p-1}}{\alpha-1}$, for an 
 irreducible $\alpha\in \Q(\sqrt{d})$ such that
 $\alpha-1$ is a unit is considered in \cite{pali:2012}. 
 Then $M_{p,\alpha}$ may be called an analog of Mersenne prime if the  norm of
 $M_{p,\alpha}$ is a rational prime.
With respect to the above condition on $\alpha$  we get only two  choices for $\alpha$, namely, $2+\sqrt{2}$ and $\sqrt{2}$. 
 For $\alpha=2+\sqrt{2}\in \Q(\sqrt{2})$,
 if the norm of $M_{p,\alpha}$ is a rational prime, then $M_{p,\alpha}$ is a quadratic residue of $7$. 
 Then it is proved that,
 $\text{ for all }p\equiv\pm1\pmod{6}$ we have $M_{p,\alpha}=x^2+7y^2$ and $x$ part is divisible by $8$.

 The objective of the paper is as follows: when does a positive integer $N$ divide either $x$ or $y$
in the representation of primes as $x^2+dy^2$. 
We consider Gaussian Mersenne primes and Eisenstein Mersenne primes to illustrate the result.
 For $p\equiv \pm 1\pmod{3}$, Gaussian Mersenne primes are represented in the form $x^2+7y^2$, then 
 we show that $x\equiv \pm 1\pmod{8}$ and $y\equiv 0\pmod{8}$. 
 Also, for $p\equiv 1\pmod{6}$ Eisenstein Mersenne primes are represented in the form 
 $x^{2}+3y^{2}$,  and $x$ leaves the remainder
 $6\pmod{7}$  and $y$ is divisible by $7$. 
 \section{Primes of the form $x^2+dy^2$ }
It is well known that,  primes represented by $x^2+dy^2$ split completely in $K=\Q(\sqrt{-d})$.
Then, for a given $N$, we see whether these primes split completely in the quartic extension $H_4$
of $\Q(\sqrt{-Nd})$. The existence of $H_4$ containing $\sqrt{N}$ and $\sqrt{-d}$ is guaranteed 
from  \cite{vau:1985:cyclqua}. Given $m$, we can find infinitely many square-free integers
$n$ such that $gcd(m,n)=1$ and $\Q(\sqrt{nm})$
has a cyclic unramified extension of degree $4$ containing $\Q(\sqrt{n})$ \cite{vau:1985:cyclqua}.
Then we see that, in the representation of primes as $x^2+dy^2$ if $N$ divides either $x$ or $y$,
then these primes indeed split in the unramified cyclic quartic extension of $\Q(\sqrt{-Nd})$.

Let $\alpha\in \Q(\sqrt{m})$ with its conjugate $\alpha^{\prime}$ and $\alpha\alpha^{\prime}=nk^2=a^2-mb^2$
and $n\neq 1,m$.
Then $\Q(\sqrt{m},\sqrt{\alpha})$ is a field of degree $4$, which is not normal over $\Q$. Its normal closure 
is $H_4=\Q(\sqrt{m},\sqrt{n},\sqrt{\alpha})=\Q(\sqrt{m},\sqrt{\alpha},\sqrt{\alpha^{\prime}})$, which
 has degree 8 and is normal over $\Q$ with dihedral Galois group and $H_4$ contains $3$ quadratic subfields 
 $\Q(\sqrt{m})$, $\Q(\sqrt{n})$ and $\Q(\sqrt{mn})$.
 
 \begin{theorem}\label{vau1}\cite{vau:1985:cyclqua}
  Let $k\in \Z$ be square-free. If there is a field $K$ of degree $8$, normal over $\Q$,
  and unramified over $\Q(\sqrt{k})$ with dihedral Galois group, then 
  \begin{enumerate}
   \item $k=mn$, $m\neq1$, $n\neq1$ and $m\equiv 1\pmod{4}$, and
   \item every prime factor of $n$ is a splitting prime in $\Q(\sqrt{m})$ (and vice versa)
  \end{enumerate}
\end{theorem}

\begin{theorem}\label{vau2} \cite{vau:1985:cyclqua} Let $m\equiv 2\pmod{4}$ and suppose $a^2-mb^2=nk^2\equiv 1\pmod{4}$
 (where $n$ is square free.)
 If $(n,m)=1,$ then $\Q(\sqrt{nm})$ has a cyclic unramified extension of degree 4 over $\Q(\sqrt{nm})$ containing 
 $\Q(\sqrt{m})$.
\end{theorem}
 Here we consider $K=\Q(\sqrt{-d})$ with $d\equiv 3\pmod{4}$ and $J=\Q(\sqrt{N})$ with $\alpha$ and $\alpha^{\prime}\in J$
 satisfying $\alpha\cdot \alpha^{\prime}=-d=a^2-Nb^2\equiv 1\pmod{4}$.
 Then
 there exists an unramified  cyclic quartic extension $H_4$ of  $S=\Q(\sqrt{-Nd})$ and $H_4$ contains two 
 quadratic subfields $K=\Q(\sqrt{-d})$ and $J=\Q(\sqrt{N})$.
 If $D=\mathrm{disc}(\Q(\sqrt{-Nd}))$ and $d_1=\mathrm{disc}(\Q(\sqrt{-d}))$ and $d_2=\mathrm{disc}(\Q(\sqrt{N}))$,
 by studying
the ramification groups of the primes $p|D$ and using the fact that they are normal
subgroups of the decomposition groups, one deduces that primes $p_1|d_1$ splits in $J$ and vice-versa.
Alternatively, $\Q(\sqrt{d_1},\sqrt{d_2})$ can be embedded into
a dihedral extension of degree $8$ \cite{lem:1985:hil}.
Above discussion yields the following lattice of fields.
 
   
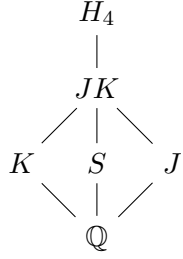
\begin{center}
  \begin{tikzpicture}
  \node (h4) at (0,2) {$H_4$};
   \node (jk) at (0,1) {$JK$};
  \node (k) at (-1,0) {$K$};
  \node (s) at (0,0) {$S$};
   \node (j) at (1,0) {$J$};
  \node (q) at (0,-1) {$\Q$};
  
  \draw (q) -- (k)  -- (jk) -- (s) -- (q) -- (j) -- (jk)-- (h4);
   
  \draw[preaction={draw=white, -,line width=6pt}] ;
\end{tikzpicture}
\end{center}

Let $\{T_n\}_{n\in \Z}$ be a form, that generates prime numbers and 
 $\{T_p\}_{p:prime}$ be sequence of prime numbers represented in the form $\{T_p\}_{p:prime}=x^2+dy^2$ 
  for some integers $x$ and $y$. Let $N>0$ be an integer that divides either $x$ or $y$.
  Now we have the following theorem.
 \begin{theorem} \label{th-split}
  Let $\{T_p\}_{p:prime}$ be a sequence of prime numbers defined as above. 
  Then $\{T_p\}_{p:prime}$ splits completely in the cyclic quartic extension $H_4$  of $S=\Q(\sqrt{-Nd})$.
\end{theorem}
\begin{proof}
The existence of cyclic quartic unramified extension of $S=\Q(\sqrt{-Nd})$  for a suitable $N$ follows from 
 Theorems \ref{vau1} and \ref{vau2}. Thus $\Q(\sqrt{-Nd})$ 
 has a cyclic unramified extension $H_4$ of degree 4 over $\Q(\sqrt{-Nd})$ containing $\Q(\sqrt{N}).
 \text{ Hence }\sqrt{N}\in H_4$. Denote $J=\Q(\sqrt{N})$ and $K=\Q(\sqrt{-d})$.
 Then we have $J\subset  H_4$ and $\mathrm{Gal}(H_4/\Q)$ is isomorphic to the 
 dihedral group with $8$ elements. So there are two conjugate field extensions of $J$,
 say $J_1$ and $J_2$ contained in $H_4$. 
 The ring of algebraic integers $\mathcal{O}_{JK}$ is a free module over $\mathcal{O}_{J}$
 and the discriminant of $JK/J=-d$ and the discriminants $\Delta(J_1/J)$ and $\Delta(J_2/J)$ must be relatively prime; 
 if not there would be a prime of $J$
 ramified in $J_1$ and $J_2$. This prime is ramified in $JK$ too, as discriminant $\Delta(H_4/JK)=1$. 
 This implies inertia field of this prime equals $J$,
 which is not possible, since discriminant of $H_4/JK=1$.
 Now, $J\subset\ JK \subset\ H_4$ and $J_1$ and $J_2$ are linearly disjoint extension of $J$
 and $-d=\Delta(J_1/J)\Delta(J_2/J)$.
 Now consider $t_p$ and $\tilde{t_p}$ in $J_1$ and $J_2$ such that $-d=t_p \tilde{t_p}$. 
 
 We apply Artin map to show that $\{T_p\}_{p:prime}$ splits completely in $H_4$. Then $\forall \sigma \in \mathrm{Gal}(J/\Q)$
 we  have $\sigma(t_p)>0$ and $\sigma(\tilde{t_p})>0$. 
 Clearly $\{T_p\}_{p:prime}$ splits completely in $\Q(\sqrt{-d})$ hence  $\left( \frac{-d}{T_p} \right)=1$.
 We consider the case when $\{T_p\}_{p:prime}\equiv 1\pmod{d}$,  other cases are dealt similarly.  
 Here we have \(t_p\equiv \tilde{t_p} \equiv 1\pmod{d} \text{ and } \forall \sigma\in \mathrm{Gal}(J/\Q) 
 \text{ we have } \sigma(t_p)>0 \text{ and }\sigma(\tilde{t_p})>0  \).
 Since $\Delta(J_1/J)$ and $\Delta(J_2/J)$ are relatively prime, we get
 \(t_p\equiv \tilde{t_p} \equiv 1\pmod{\Delta(J_i/J)}\mathcal{O}_{J} \text{ for } i=1,2 \)
 and  \(
 \forall \sigma\in \mathrm{Gal}(J/\Q) \text{ we have } 
 \sigma(t_p)>0 \text{  and  } \sigma(\tilde{t_p})>0  \).
 With these conditions above we apply Artin map and we get \(
 \left( \frac{t_p}{J_i/J} \right)=\left( \frac{\tilde{t_p}}{J_i/J} \right)=1 \text{ for } i= 1,2.\)
 The Galois group of $H_4/J$ is isomorphic to $\mathrm{Gal}(J_1/J)\times \mathrm{Gal}(J_2/J)$. 
 Hence $\left( \frac{t_p}{H_4/J} \right)=\left( \frac{\tilde{t_p}}{H_4/J} \right)=1$.
 Thus $\{T_p\}_{p:prime}$ splits completely in $H_4$.
 \end{proof}
 In Theorem \ref{th-split}, we only need to prove that, primes  $\{T_p\}_{p:primes}=x^2+dy^2$ split completely,
 in the cyclic unramified extension  $H_4$ of $\Q(\sqrt{-Nd})$,  whenever the $x$ part or $y$
 part is divisible by $N$.
 The cyclic quartic unramified extension, $H_4$ of $\Q(\sqrt{-Nd})$ is the 
Hilbert class field  $H(\Q(\sqrt{-Nd}))$  only when
$\mathrm{Cl}(\Q(\sqrt{-Nd}))$ is the cyclic group with four elements. 
The cyclic quartic extension $H_4\subset H(\Q(\sqrt{-Nd})$ only when $4$ divides the class number of $\Q(\sqrt{-Nd})$.

In the following few examples we see that the quartic extension $H_4$ of $\Q(\sqrt{-Nd})$ is the Hilbert class field
 of $\Q(\sqrt{-Nd})$.
 For $d=-7$ and $N=8$, $H_4$ is the Hilbert class field of $\Q(\sqrt{-14})$.
 For $d=-3$ and $N=7$, $H_4$ is the Hilbert class field of $\Q(\sqrt{-21})$.
 For $d=-7$ and $N=29$, $H_4$ is the Hibert class field of $\Q(\sqrt{-203})$.

 But for $d=-31$ and $N=8$, $H_4$ exists and is {\it not} the Hilbert class field of $\Q(\sqrt{-62})$, but 
 $H_4\subset H(\Q(\sqrt{-62}))$. The Hilbert class
 field of $\Q(\sqrt{-248})$ is $x^8- 3x^7 + 3x^6 - 2x^5 + 2x^4 + 2x^3 + 3x^2 + 3x + 1$ \cite{Cohen:2000:atcnt}.
 From Theorem \ref{vau2}, we have, \(1-2\cdot4^2= -31\). 
 Hence $\Q(\sqrt{-62})$ has a cyclic unramified extension of degree $4$
 and $4$ divides the class number. For example, $2147483647=5176^{2}+31\cdot(8271)^{2}$ and $5176$ is divisible by $8$.

With the above choice of $N$ and $d$, we have,
$\mathrm{Gal}(H_4/\Q)$ is isomorphic to the dihedral group with $8$ elements, and 
there are two conjugate field extension of $K=\Q(\sqrt{-d})$, say $K_1$ and $K_2$.
 Then, $K_1$, $K_2$ and $J_1$, $J_2$ are two disjoint conjugate 
 field extesions of $K$ and $J$ respectively. Thus,
 we have the lattice of fields as given below.
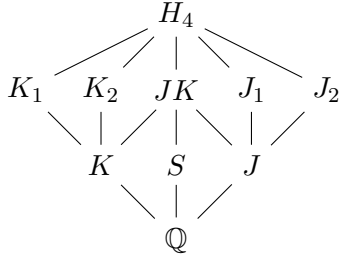
\begin{center}
  \begin{tikzpicture}
  \node (h4) at (0,2) {$H_4$};
  \node (k1) at (-2,1) {$K_1$};
  \node (k2) at (-1,1) {$K_2$};
  \node (jk) at (0,1) {$JK$};
  \node (j1) at (1,1) {$J_1$};
  \node (j2) at (2,1) {$J_2$};
  \node (k) at (-1,0) {$K$};
  \node (s) at (0,0) {$S$};
   \node (j) at (1,0) {$J$};
  \node (q) at (0,-1) {$\Q$};
  
  \draw (q) -- (k) -- (k1) -- (h4) -- (k2) -- (k) -- (jk) -- (s) -- (q) -- (j) -- (jk)-- (h4) -- (j1) -- (j)-- (j2) -- (h4);
   
  \draw[preaction={draw=white, -,line width=6pt}] ;
\end{tikzpicture}
\end{center}

 Since  $K\subset JK\subset H_4$, we have, $\Delta(K_i/K)|\Delta(H_4/K)$ for $i\in\{1,2\}$,
 and $\Delta(K_1/K)$ and $\Delta(K_2/K)$ are coprime.
  Then we have the surjective homomorphism $\mathrm{Cl}_{\Delta}(K)\twoheadrightarrow \mathrm{Gal}(K_1/K)$. 
  The group $I_{K}(\Delta)/P_{K,1}(\Delta)$, is a subgroup of $\mathrm{Cl}_{\Delta}(K)$.
   The group $I_{K}(\Delta)/P_{K,1}(\Delta)$ contains the principal prime ideal $\pi=p_1+q_2\sqrt{-d}$ of $K$.
  Now consider a rational prime $p$, that does not split in $J$ and $K$, but splits in $S$, 
  so that the decomposition field of a prime $\mathfrak{p}$ in $H_4$ above $p$
  is $S.$ Thus the prime $p\mathcal{O}_K$ of $K$ does not split in $K_1$, and 
  therefore $I_{K}(\Delta)/P_{K,1}(\Delta)$ maps surjective on $\mathrm{Gal}(K_1/K)$. 
  Then we have the group homomorphism 
  
  $(\Z/N\Z)^{*}\simeq(\mathcal{O}_{K}/\Delta)^{*}\twoheadrightarrow I_{K}(\Delta)/P_{K,1}(\Delta)\twoheadrightarrow 
  \mathrm{Gal}(K_1/K)$
  By Theorem \ref{th-split}, we have
  $\{T_p\}_{p:prime}$ splits completely in $H_4$, hence $\left(\frac{\pi}{K_1/K}\right)=1$. 
  Hence $\pi=(x+y\sqrt{-d})$ is the identity element in 
  $I_{K}(\Delta)/P_{K,1}(\Delta)$, thus $x+y\sqrt{-d}\equiv \pm1\pmod{\Delta}$.
  
  \section{Gaussian Mersenne primes in the form $x^2+dy^2$} \label{Gp}
A good introduction to Gaussian Mersenne primes and Eisenstein Mersenne primes  can be found 
in \cite{berriziskra:2010:eisenmersenneprimes}.   
A Gaussian Mersenne number is an element of $\Z[i]$
given  by $\mu_p=(1\pm i)^{p}-1$, for some rational prime $p$.
The Gaussian Mersenne norm  $G_p=2^{p}-\left(\frac{2}{p}\right)2^{\frac{p+1}{2}}+1$ 
is  the norm of $\mu_p$,  $N(\mu_p)$. If $G_p$ is a 
 rational prime then we call  $G_p$  a Gaussian Mersenne prime. 
  Let $R_m(K)$ be the ray class field of $K=\Q(\sqrt{-d})$ with modulus $m$ and $H(K)$ is the Hilbert class field of $K$.
  Here we state two properties of Gaussian Mersenne primes when represented in the form $x^2+dy^2$. 
\begin{proposition}\label{drem3}
 Let $d\equiv 3\bmod{4}$ be a square-free integer. 
 
 Suppose $L=R_{2}(\Q(\sqrt{-d}))(\sqrt{2})=H(\Q(\sqrt{-2d}))$.
 $G_{p}$ be a Gaussian Mersenne prime unramified in $L$ then  
 $G_{p}=x^2+dy^2$ if and only if  $G_{p}=x_1^{2}+2 dy_1^{2}$.
\end{proposition}
\begin{proof}
 We have
 $G_{p}=2^{p}-\left(\frac{2}{p}\right)2^{\frac{p+1}{2}}+1$. For
 $p>3$, $G_{p}\equiv1\bmod{8}$. Hence, 
$G_{p}$ splits completely in $\Q(\sqrt{2})$ and 
the prime $G_{p}$ is unramified in $L$,
 which implies $G_{p}\nmid 2d$. We have,
 $[\mathcal{O}_{K}:\Z[\sqrt{-d}]]=2$, since $d\equiv 3\pmod{4}$ 
is a positive square-free integer. Now, by using the fact $G_{p}$
splits in $\Q(\sqrt{2})$,   
  $G_{p}=x^2+dy^2$ if and only if $G_{p}$ splits completely in
 $R_{2}(\Q(\sqrt{-d}))(\sqrt{2})$. 
 For $K=\Q(\sqrt{-2d})$ we have $[\mathcal{O}_{K}:\Z[\sqrt{-2d}]=1$. 
Thus,  $G_{p}=x_1^{2}+2dy_1^2$ if
 and only if $G_{p}$ splits completely in $H(\Q(\sqrt{-2d}))$.
 By assumption we have
 $R_{2}(\Q(\sqrt{-d}))(\sqrt{2})=H(\Q(\sqrt{-2d}))$ hence the result. 
 \end{proof}
 
 \begin{proposition}\label{drem1}
 Let $d\equiv 1\bmod{4}$ be a square-free integer. 
 
 Suppose $L=H(\Q(\sqrt{-d}))=H(\Q(\sqrt{-2d}))$. Let
 $G_{p}$ be a Gaussian Mersenne prime unramified in $L$  then  
 $G_{p}=x^2+dy^2$ if and only if  $G_{p}=x_1^{2}+2 dy_1^{2}$.  
\end{proposition}
\begin{proof}
 Proof follows from the previous proposition, except for the fact,  $[\mathcal{O}_{K}:\Z[\sqrt{-d}]]=2$,
 which is $1$ in the present case, as $d\equiv 1\bmod{4}$.
\end{proof}
 Now we give first few examples of Gaussian Mersenne primes and their representation as $x^2+7y^2$.
 
 \(G_{7}=113=1+7\cdot4^{2}\) 
\(G_{47}=140737471578113= 5732351^{2}+7\cdot3925696^{2}\) 

\(G_{73}= 9444732965601851473921=96890022433^{2}+7\cdot2854983576^{2} \)

$G_{113}=10384593717069655112945804582584321=$\\ \(79288509938147361^{2}+7\cdot24195412519312600^{2} \)

For $p>7$, a  simple calculation  shows that, in each case $x\equiv \pm 1 \bmod{8}$ and $y\equiv 0\bmod{8}$.  
Also, from the above table it is not difficult to see  that $y$ is exactly  divisible by $4$.
 These two observations are stated as a lemma below.
 \begin{lemma}\label{mainlemma}
  If  $G_{p}$ is represented in the form  $x^2+7y^2$ then  $x\equiv \pm1\pmod{8}$ and $4|y$.
 \end{lemma}
 \begin{proof}
  Proof is a simple modular arithmetic calculation.
 \end{proof}

 Now it remains to prove that $y$ is divisible by $8$ in the expression $G_p=x^2+7y^2$.
 As a first step we  show that $G_p$ splits completely in the cyclic quartic unramified extension of $\Q(\sqrt{-14})$.

\begin{theorem} \label{mainprop}
  Let $p>7$ and $G_p=2^{p}-(\frac{2}{p})2^{\frac{p+1}{2}}+1$ be a Gaussian Mersenne prime. Then for $p\equiv\pm 1\pmod{3}$
 the form $G_p= x^2+7y^2$ exists. Then there exists a cyclic extension $H_4$ of
$S=\Q(\sqrt{-14})$ with $[H_4:S]=4$, $H_4\subset H(S)$ and
$\sqrt{2}\in H_4$ and  $G_p$ splits completely in $H_4$.
 \end{theorem}

 Proof easily follows from Theorem \ref{th-split}.

 From Lemma~\ref{mainlemma} and Theorem~\ref{mainprop} it is clear
that  $G_{p}$ splits completely in the cyclic extension of
$\Q(\sqrt{-14})$. 
Now we are ready to prove our main result that, for  $p>7$ and  $p\equiv\pm 1\pmod{3}$ in the representation of $G_p$ as
$G_p=x^2+7y^2$ for some $x,y\in \Z$,  $y$ is divisible by $8$.
\begin{theorem} \label{mainth}
 Let $p>7$ and $G_p=2^{p}-(\frac{2}{p})2^{\frac{p+1}{2}}+1$ be a Gaussian Mersenne prime. Then for $p\equiv\pm 1\pmod{3}$
 the form $G_p=x^2+7y^2$ exists for some $x,y\in \Z$ and $y$ is divisible by $8$.
\end{theorem}
\begin{proof}
Taking $J=\Q(\sqrt{2})$, and $K=\Q(\sqrt{-7})$, we have $S=\Q({\sqrt{-14}})$ and 
 $\omega=\frac{1+\sqrt{-7}}{2}$. Let $\bar \omega$  be the conjugate of
 $\omega$ then $\omega\bar\omega=2$.
$2$ splits completely in $K$ as $2=(2,\omega)(2,\bar\omega)$ and  
$\Delta(K_1/K) \cdot \Delta(K_2/K)=8$.
Since the discriminants are coprime and $K_1$ and $K_2$  are conjugates, $\Delta(K_1/K)=(2,\omega)^{3}$ and
$\Delta(K_2/K)=(2,\bar\omega)^{3}$. Now, $\Delta(H_4/\Q)=8^2$ and $\Delta(K_i/K)|8^2$.
Let $\Delta=\Delta(K_1/K)$. Then there is a
surjective group homomorphism 
\[\mathrm{Cl}_{\Delta}(K)\twoheadrightarrow \mathrm{Gal}(K_2/K).\]
The group $I_{K}{(\Delta)}/P_{K,1}(\Delta)=P_{K}(\Delta)/P_{K,1}(\Delta)$ is a subgroup of $\mathrm{Cl}_{\Delta}(K)$ and 
contains the principal prime ideal $\pi=(x+y\sqrt{-7})$ of $K$.
Any prime of $\Q$, which is inert in both $J$  and $K$ splits in $S$. For example, since $7\equiv1\pmod{3}$,
the prime $3$ of $\Q$ is inert in $J$ and $K$ and splits in $S$. Thus
the decomposition field of a prime $\mathfrak{p}_3$ in $H_4$ above $3$ is $S$. 
The prime $3\mathcal{O}_{K}$ of $K$ does not split in $K_1$. Hence, from the above discussion,
$P_{K}(\Delta)/P_{K,1}(\Delta)$ maps surjective on $\mathrm{Gal}(K_1/K)$. 
The map from $(\mathcal{O}_{K}/\Delta)^{*}$ to $P_{K}(\Delta)/P_{K,1}(\Delta)$
is surjective and  the group $(\mathcal{O}_{K}/\Delta)^{*}$ is isomorphic to the group
$(\Z/8\Z)^{*}$. We have the following group homomorphism,
\[(\Z/8\Z)^{*}\simeq(\mathcal{O}_{K}/\Delta)^{*}\twoheadrightarrow P_{K}(\Delta)/P_{K,1}(\Delta)\twoheadrightarrow \mathrm{Gal}(K_1/K).\]
Now $\{\pm 1\}$ is contained in the kernel of the $(\mathcal{O}_{K}/\Delta)^{*}\twoheadrightarrow P_{K}(\Delta)/P_{K,1}(\Delta)$ and this map is surjective hence,
the kernel of $(\mathcal{O}_{K}/\Delta)^{*}\rightarrow P_{K}(\Delta)/P_{K,1}(\Delta)$ is $\{\pm 1\}$. Thus $P_{K}(\Delta)/P_{K,1}(\Delta)$ has two elements.
So,\[(\Z/8\Z)^{*}\simeq(\mathcal{O}_{K}/\Delta)^{*}\rightarrow P_{K}(\Delta)/P_{K,1}(\Delta)\simeq \mathrm{Gal}(K_1/K).\]
From Theorem~\ref{mainprop} it is clear that $G_{p}$ splits
completely in $H_4$ and $H_4$ contains $\sqrt{2}$. 
For any prime ideal $\pi$ of $K$, $\left(\frac{K_1/K}{\pi}\right)=1$. Hence, $\pi=x+y\sqrt{-7}$ is identity in $P_{K}(\Delta)/P_{K,1}(\Delta)$.
Thus,
\begin{equation}
  \label{maineq}x+y\sqrt{-7}\equiv\pm 1\bmod{\Delta}.
\end{equation}
From Lemma~\ref{mainlemma}, we have $4|y$. We assume that, $p>7$, so
$G_{p}\equiv 1\pmod{32}$. Since, $4|y$ and $x\equiv\pm 1\pmod{8}$, we have,
\[x^2+7y^2\equiv 1\pmod{16}\] and \[ (y\cdot\sqrt{-7})^{2}\equiv 0\pmod{16}.\]
That is $y\cdot\sqrt{-7}\equiv 0 \text{ or }4\pmod{8}$. 
 If $y\cdot\sqrt{-7}\equiv   4\bmod{8}$, then there is a contradiction to 
 $\pi\equiv x+y\sqrt{-7}\equiv\pm 1\bmod{\Delta}$ in \eqref{maineq}
 and to the splitting of $\pi$ in $H_4$.  Hence, \begin{equation}\label{eq1}
                                                  y\cdot\sqrt{-7}\equiv 0  \bmod{8}   
                                                 \end{equation}
Now, for    $K=\Q(\sqrt{-7})$, let $\mathcal{O}_K$ be its ring of integers. The element   $\omega=\frac{1+\sqrt{-7}}{2}$ is an
integer of $K$, which  is a zero of the polynomial $X^2-X+2$.
As $-2\in \Z/8\Z$ satisfies $X^2-X+2=0$, we obtain a  ring homomorphism  
\[ \mathcal{O}_{K}\mapsto \Z/8\Z \text{ mapping }
 a+b\omega\mapsto (a-2b)\bmod{8}.\] 
Now $\sqrt{-7}=2\omega-1$ maps to $3$, confirming $y\equiv 0\pmod{8}$ from equation \ref{eq1}. This completes the proof.
\end{proof}
  
\section{ Eisenstein Mersenne primes in the form  $x^2+3y^2$}
For a positive integer $n$, the $n$th Eisenstein Mersenne norm is defined by $E_n=N((1-\omega)^{n}-1)$, where $\omega$ is a cube root of $1$.
If for a prime suffix $p$, $E_p$ is a rational prime, then $E_p$ is called a Eisenstein Mersenne prime, 
we have $E_p=3^{p}-\left(\frac{3}{p}\right)3^{\frac{p+1}{2}}+1$. 
Here $\left(\frac{\cdot}{p}\right)$ is the Legendre symbol $\bmod{p}$.
For $p\equiv 1\pmod{6}$,  Eisenstein Mersenne primes are represented in the form $x^2+3y^2$, then $x$ leaves the remainder
 $6\pmod{7}$  and $y$ is divisible by $7$.
Here we list first few  Eisenstein Mersenne primes as $x^{2}+3y^2$\\
  $E_7=2269=41^{2}+3\cdot 14^{2}$\\
  $E_{19}= 1162320517=  29525^{2}+3\cdot9842^{2}$\\
  $E_{79}=49269609804781974450852068861184694669=\\
  6078832729528464401^{2}+3\cdot2026277576509488134^{2}$\\
   In the above representation we notice that, $41,29525, 6078832729528464401$ are congruent to $6\bmod{7}$ and  
  $14, 9842, 2026277576509488134$ are divisible by $7$.
  These results can be proved not without much effort similarly as discussed for Gaussian Mersenne primes, 
  by constructing unramified cyclic quartic extensions of $\Q(\sqrt{-21})$.
\begin{theorem}
 Let $p\geq7$ and $E_p$ be an Eisenstein Mersenne prime. Then for $p\equiv 1\pmod{6}$
 the form $E_p= x^2+3y^2$ exists. Then there exists a cyclic extension $H_4$ of
$S=\Q(\sqrt{-21})$ with $[H_4:S]=4$, $H_4\subset H(S)$ and
$\sqrt{7}\in H_4$ and  $E_p$ splits completely in $H_4$.
\end{theorem}
  \begin{proof}
   Proof easily follows from Theorem \ref{th-split}.
  \end{proof}

  \begin{theorem}
   Let $p\geq7$ and $E_p$ be an Eisenstein Mersenne prime. Then for $p\equiv 1\pmod{6}$
 the form $E_p= x^2+3y^2$ exists. Then $y$ is divisible by $7$.
  \end{theorem}
\begin{proof}
 Taking $J=\Q(\sqrt{7})$, and $K=\Q(\sqrt{-3})$, we have $S=\Q({\sqrt{-21}})$ and 
 $\omega=\frac{1+3\sqrt{-3}}{2}$. Let $\bar \omega$  be the conjugate of
 $\omega$ then $\omega\bar\omega=7$. Hence
$7$ splits completely in $K$  and  
$\Delta(K_1/K) \cdot \Delta(K_2/K)=28$.
Now, $\Delta(H_4/\Q)=28^2$ and $\Delta(K_1/K)|\Delta(H_4/\Q)$ and $\Delta(K_2/K)|\Delta(H_4/\Q)$.
Let $\Delta=\Delta(K_1/K)$. Then there is a
surjective group homomorphism:\[\mathrm{Cl}_{\Delta}(K)\twoheadrightarrow \mathrm{Gal}(K_2/K).\] 
The group $I_{K}{(\Delta)}/P_{K,1}(\Delta)=P_{K}(\Delta)/P_{K,1}(\Delta)$ is a subgroup of $\mathrm{Cl}_{\Delta}(K)$ and 
contains the principal prime ideal $\pi=(x+y\sqrt{-3})$ of $K$.
The prime $11$ of $\Q$,  is inert in both $J$  and $K$ hence splits in $S$.   Thus
the decomposition field of a prime $\mathfrak{p_{11}}$ in $H_4$ above $11$ is $S$. 
The prime $11\mathcal{O}_{K}$ of $K$ does not split in $K_1$. Hence,
$P_{K}(\Delta)/P_{K,1}(\Delta)$ maps surjective on $\mathrm{Gal}(K_1/K)$. 
The map from $(\mathcal{O}_{K}/\Delta)^{*}$ to $P_{K}(\Delta)/P_{K,1}(\Delta)$
is surjective and  the group $(\mathcal{O}_{K}/\Delta)^{*}$ is isomorphic to the group
$(\Z/28\Z)^{*}$. We have the following group homomorphism,
\[(\Z/28\Z)^{*}\simeq(\mathcal{O}_{K}/\Delta)^{*}\twoheadrightarrow P_{K}(\Delta)/P_{K,1}(\Delta)\twoheadrightarrow \mathrm{Gal}(K_1/K).\]
For any prime ideal $\pi$ of $K$, $\left(\frac{K_1/K}{\pi}\right)=1$. Hence, $\pi=x+y\sqrt{-3}$ is identity in $P_{K}(\Delta)/P_{K,1}(\Delta)$.
Thus,
\begin{equation}
  x+y\sqrt{-3}\equiv\pm 1\bmod{\Delta}.
\end{equation}
Clearly we obtain \[x^2+3y^2\equiv 1\pmod{14}\quad \forall p\equiv 1\pmod{6}\]
and then \(x^2+3y^2\equiv 1\pmod{2}\)and $x^2+3y^2\equiv 1\pmod{7}$ $\forall p\equiv 1\pmod{6}$. 
Then a simple calculation leads to $x^2\equiv 1\pmod{7}$, hence confirming $y\equiv 0\pmod{7} $. This completes the proof.

\end{proof}

\section{Conclusion}
The problem of representing primes in the form $x^2+dy^2$ is well studied in \cite{Cox:1989:pfct}. 
Here we have found that, if a positive integer $N$ divides either $x$ or $y$ in the representation of 
primes as $x^2+dy^2$, then those primes split completely in the unramified cyclic quartic
extension $H_4$ of $\Q(\sqrt{-Nd})$.  As an example we have shown that
when Gaussian Mersenne primes are represented in the form $x^2+7y^2$, 
$y$ is divisible by $8$. Also, when Eisenstein Mersenne primes are represented in the form $x^2+3y^2$,
$y$ is divisible by $7$. 
\section{Acknowledgement}
The authors thank  the referee for carefully reading through the
manuscript and for helpful suggestions.
Special thanks to Professor B.R. Shankar, Department of MACS, NITK, Surathkal
for his useful comments on the draft.
The second author  sincerely  thanks  National Board for Higher Mathematics (NBHM), 
Government  of India for funding this work through Post Doctoral Research Fellowship. 


\end{document}